\documentclass[11pt]{amsart}

\usepackage[colorlinks=true, pdfstartview=FitV, linkcolor=blue, 
citecolor=blue]{hyperref}

\usepackage{a4wide}
\usepackage{amsmath,amssymb} 
\usepackage{bbm}

\newenvironment{bsmallmatrix}
  {\left[\begin{smallmatrix}}
  {\end{smallmatrix}\right]}

\theoremstyle{plain}
\newtheorem{theorem}{Theorem}[section]

\newtheorem{lemma}[theorem]{Lemma}
\newtheorem{coro}[theorem]{Corollary}

\theoremstyle{definition}

\newtheorem{question}[theorem]{Question}
\newtheorem*{lproblem}{Lehmer's problem}
\newtheorem{example}[theorem]{Example}
\newtheorem{remark}[theorem]{Remark}

\newcommand{\dd}{\,\mathrm{d}}
\newcommand{\ii}{\ts\mathrm{i}\ts}
\newcommand{\ee}{\mathrm{e}}
\newcommand{\ts}{\hspace{0.5pt}}
\newcommand{\nts}{\hspace{-0.5pt}}

\newcommand{\fm}{\mathfrak{m}\ts}

\newcommand{\ZZ}{\mathbb{Z}}

\newcommand{\RR}{\mathbb{R}}
\newcommand{\NN}{\mathbb{N}}
\newcommand{\CC}{\mathbb{C}}
\newcommand{\TT}{\mathbb{T}}

\newcommand{\cC}{\mathcal{C}}
\newcommand{\cV}{\mathcal{V}}

\newcommand{\vS}{\varSigma}

\newcommand{\bs}{\boldsymbol}

\newcommand{\exend}{\hfill $\Diamond$}

\newcommand{\myfrac}[2]{\frac{\raisebox{-2pt}{$#1$}}
      {\raisebox{0.5pt}{$#2$}}}

\begin{document}

\title[Binary substitutions and Mahler measures]
{Binary constant-length substitutions and \\ [1mm] 
Mahler measures of Borwein polynomials}


\author{Michael Baake}
\address{Fakult\"at f\"ur Mathematik, Universit\"at Bielefeld, \newline
\hspace*{\parindent}Postfach 100131, 33501 Bielefeld, Germany}
\email{$\{$mbaake,cmanibo$\}$@math.uni-bielefeld.de }

\author{Michael Coons}
\address{School of Mathematical and Physical Sciences, 
   University of Newcastle, \newline
\hspace*{\parindent}University Drive, Callaghan NSW 2308, Australia}
\email{michael.coons@newcastle.edu.au}

\author{Neil Ma\~nibo} 

\dedicatory{In memory of Jonathan Michael Borwein (1951--2016)}

\begin{abstract}  
  We show that the Mahler measure of every Borwein
  polynomial{\ts}---{\ts}a polynomial with coefficients in
  $ \{-1,0,1 \}$ having non-zero constant term{\ts}---{\ts}can be
  expressed as a maximal Lyapunov exponent of a matrix cocycle that
  arises in the spectral theory of binary constant-length
  substitutions. In this way, Lehmer's problem for height-one
  polynomials having minimal Mahler measure becomes equivalent to a
  natural question from the spectral theory of binary constant-length
  substitutions. This supports another connection between Mahler measures
  and dynamics, beyond the well-known appearance of Mahler measures as
  entropies in algebraic dynamics.
\end{abstract}

\maketitle

\section{Introduction}

Let $p$ be a polynomial with complex coefficients. The
\emph{logarithmic Mahler measure} of $ p $ is defined to be the
logarithm of the geometric mean of $\lvert p \rvert$ over the unit
circle; that is,
\begin{equation}\label{eq:def-M}
   \fm(p) \, :=  \int_0^1\log \, \bigl| \ts
    p \bigl( \ee^{2\pi \ii t} \bigr) \bigr| \dd t \ts .
\end{equation}
It is well known and easily shown using Jensen's formula 
\cite[Prop.~16.1]{KS-Book} that the logarithmic Mahler
measure satisfies
\[
   \fm(p) \, = \, \log | a^{}_s| \, + \sum_{j=1}^{s} 
  \log \bigl(\max\{|\alpha^{}_j|,1\} \bigr),
\]
where $p(z)=a^{}_s \ts\prod_{j=1}^s (z-\alpha^{}_j)$; see \cite{EvWa}
for background. Here, we will only consider integer polynomials.
Kronecker's lemma \cite{K1857} then implies that $\fm(p)=0$ if and
only if $p$ is a product of a cyclotomic polynomial (not necessarily
irreducible) and a monomial. In this way, $\fm$ is a measure of the
distance of an integer polynomial to the unit circle.

One of the most interesting and long-standing problems in this area
concerns finding polynomials with small logarithmic Mahler
measures. Lehmer found the polynomial
\begin{equation}\label{eq:Lehmer}
      \ell^{}_{\mathrm{L}} (z) \, = \, 1+z-z^3-z^4-z^5-z^6-z^7+z^9+z^{10},
\end{equation}
which is irreducible and has precisely one root outside the unit
disk. This root is real and a Salem number.  The polynomial
$\ell^{}_{\mathrm{L}}$ is the polynomial with the smallest known
positive logarithmic Mahler measure,
\[
    \fm (\ell^{}_{\mathrm{L}}) \, \approx \, \log (1.176 {\ts} 281)  \ts .
\] 

\begin{lproblem} 
  Does there exist a constant $c>0$ such that any irreducible
  non-cyclotomic polynomial $p$ with integer coefficients satisfies
  $\fm(p)\geqslant c$ ?
\end{lproblem}

There are some special classes of polynomials for which Lehmer's
problem has long been \mbox{answered} in the affirmative. In
particular, there is a very interesting gap result for non-reciprocal
polynomials due to Smyth \cite{Smy71}; see also Breusch
\cite{Bre1951}. A polynomial $p$ is \emph{reciprocal} (in the wider
sense) if $p(z)= \pm z^{\deg (p)} p(1/z)$; that is, a polynomial is
reciprocal if its sequence of coefficients is palindromic, up to an
overall sign.  It follows from Smyth's result that, for non-reciprocal
polynomials, one either has $\fm (q) = 0$ or
$\fm (q) \geqslant \log (\lambda_{\mathrm{p}})$, where
$\lambda_{\mathrm{p}}$ is the smallest Pisot number, which is the real
root of $z^3 - z - 1$, also known as the \emph{plastic number}
\cite[Ex.~2.17]{TAO}.  Specialising this class a bit more, Borwein,
Hare and Mossinghoff \cite[Cor.~1.2]{BHM} showed that all
non-reciprocal polynomials $q$ with odd integer coefficients satisfy
the bound
\[
    \fm (q) \, \geqslant \, \fm (z^2 - z - 1)
   \, = \, \log (\tau) \ts ,
\] 
where $\tau = \frac{1}{2} \bigl( 1 + \sqrt{5}\, \bigr)$ is the golden
ratio, a well-known Pisot number. The golden ratio is characterised by
the property that it is the smallest limit point of Pisot numbers.
See \cite{Smyth} for a general survey, \cite{Boyd} for work
on reciprocal polynomials,  and \cite{BoydM,ERS,Moss1998,MRW2008} 
for more results on small Mahler measures and limit points.

One of the most studied classes of integer polynomials in relation to
Lehmer's problem are the \emph{Borwein
  polynomials}{\ts}---{\ts}polynomials of height $1$ (coefficients in
$\{-1,0,1\}$) with non-zero constant term; see \cite{DJS}. Special
importance is placed on this class, since for any integer polynomial
$p$ with $\fm (p) < \log (2)$ there is an integer polynomial $q$ such
that $p\, q$ has {height~$1$}; see Pathiaux \cite{P1972}. Boyd
\cite{Boyd-1} notes that, in his experience, such a $q$ can be taken
to be cyclotomic and of fairly small degree relative to the degree of
$p$; see also Mossinghoff \cite{Mossthesis}. If Boyd's observation
were proved true in general, then to solve Lehmer's problem it would
be enough to consider only Borwein polynomials; unfortunately, this is
still unknown.

Before we continue, let us mention that there is a well-known
connection between Mahler measures and algebraic dynamics. Here,
logarithmic Mahler measures show up as entropies of $\ZZ^d$-shifts of
algebraic origin \cite{EvWa,LSW,KS-Book}. The first appearance of a
Mahler measure in a similar context actually dates back to a paper by
Wannier \cite{Wannier} on the groundstate entropy of the
antiferromagnetic Ising model on the triangular lattice; see Remark
\ref{rem:Wannier} below for details. This general connection between
Mahler measures and entropy has initiated many investigations and
enhanced our knowledge about Mahler measures significantly; see
\cite{LSW,KS-Book} and the references therein for a detailed account.

In this paper, under some quite natural assumptions, we relate the
logarithmic Mahler measure of Borwein polynomials to a Lyapunov
exponent from the spectral theory of substitutions; see
\cite{Deninger-L} for an earlier appearance of a connection between
Mahler measures and Lyapunov exponents.  A \emph{binary
  constant-length substitution} $\varrho$ is defined on
$\vS_2:= \{ 0,1 \}$ by
\begin{equation*}
   \varrho: \, \begin{cases} 0 \mapsto w^{}_{0} \\ 
     1\mapsto w^{}_{1}\, , \end{cases}
\end{equation*}
where $w^{}_{0} \text{ and } w^{}_{1}$ are finite words over $\vS_{2}$
of equal length
$\lvert w^{}_{0}\rvert = \lvert w^{}_{1}\rvert = L\geqslant 2$.  Such
substitutions are important objects of research in many areas of
mathematics, ranging from dynamics and combinatorics (as
substitutions) to number theory (this is the class of binary automatic
sequences) and theoretical computer science (under the name of uniform
morphisms).

Recall that the \emph{substitution matrix} of $\varrho$ is the matrix
$M_{\varrho} = (m^{}_{ij})^{}_{0\leqslant i,j \leqslant 1}$ where 
$m^{}_{ij} \geqslant 0$ is the number of letters $i$
in the word $w^{}_j$. This matrix is also known as the
\emph{Abelianisation} of $\varrho$; compare \cite[Sec.~4.1]{TAO}. We
say that $\varrho$ is \emph{primitive} if the non-negative matrix
$M_\varrho$ is primitive, and \emph{aperiodic} if the hull (or shift)
defined by $\varrho$ does not contain any element with a non-trivial
period. When $\varrho$ is primitive, this is the case if and only if
any of the two-sided fixed points of $\varrho$ (or of $\varrho^n$ with
a suitable $n\in\NN$) with legal core (or seed) is non-periodic.  If
one of these fixed points is non-periodic, they all are, due to
primitivity; see \cite[Sec.~4.2]{TAO} for notions and further details.

Our main result is the following theorem; the relevant concepts
concerning Lyapunov exponents are recalled in Section \ref{sec:Lya}.

\begin{theorem}\label{main} 
  For any primitive, binary constant-length substitution\/ $\varrho$,
  the extremal Lyapunov exponents are explicitly given by
\[
  \chi^{B}_{\min}  \, = \, 0 
  \quad \text{and} \quad
  \chi^{B}_{\max}  \, = \, \fm  (p_\varrho) \ts ,
\]
where\/ $p_\varrho$ is a Borwein polynomial, easily determined by\/
$\varrho$. In particular, if $p_{\varrho}$ is non-reciprocal, one
has\/ $\chi^{B}_{\max} \geqslant \log (\lambda_{\mathrm{p}})$, where\/
$\lambda_{\mathrm{p}}$ is the plastic number.
\end{theorem}

Theorem~\ref{main} is a statement relating the logarithmic Mahler
measure of Borwein polynomials to Lyapunov exponents of binary
constant-length substitutions. Depending on which object one is
interested in, it can be used in a couple of different ways. As it
reads above, if one has a binary substitution, one can easily compute
the extremal Lyapunov exponents using the associated Borwein
polynomial. Alternatively, if one has a Borwein polynomial, one can
determine an associated binary constant-length substitution. This
relationship can be exploited to give some general results about
extremal Lyapunov exponents for certain binary substitutions. For
example, one now has a rather general result considering
\emph{bijective} substitutions, which are the substitutions where the
letters in the words $w^{}_0$ and $w^{}_1$ are different at each
position. Now, \cite[Cor.~1.2]{BHM} in conjunction with
Lemma~\ref{lem:subst-per} below implies the following consequence of
Theorem~\ref{main}; see Example~\ref{ex:Little} for more details.

\begin{coro} 
  Suppose that the primitive, binary constant-length substitution\/
  $\varrho$ is bijective, and that\/ $w^{}_0$ is neither a palindrome
  nor an anti-palindrome. Then, $\chi^{B}_{\max}\geqslant\log (\tau)$,
  where\/ $\tau$ is the golden ratio.
\end{coro}

As it turns out, primitive, binary constant-length substitutions which
are periodic do not satisfy the assumptions of this corollary; in
fact, for such periodic substitutions one has that
$\chi^{B}_{\max}=0$. A characterisation and further details regarding
these periodic substitutions are given later; see Lemma
\ref{lem:subst-per} and Theorem \ref{thm:per0}.

In view of the above results, in the case of Borwein polynomials,
Lehmer's problem can be restated in terms of the Lyapunov exponents
for binary constant-length substitutions.

\begin{lproblem}[dynamical analogue] 
  Does there exist a constant $c>0$ such that, for any primitive,
  binary constant-length substitution $\varrho$ with
  $\chi^{B}_{\max}\neq 0$, we have
  $\chi^{B}_{\max}\geqslant c \,$?
\end{lproblem}

\begin{remark} 
  Strong versions of both Lehmer's problem and our dynamical analogue
  would ask whether the constant $c$ can be taken to be 
  $\fm(\ell^{}_{\mathrm{L}})$, the logarithmic Mahler measure of 
  Lehmer's polynomial from Eq.~\eqref{eq:Lehmer}.
  \exend
\end{remark}

Viewing Lehmer's problem in a dynamical setting, as related to
constant-length substitutions, has some heuristic benefits. In this
area, especially at the interface with number theory, gap results are
common and expected. For example, if $f(n)$ denotes the
$n^{\text{th}}$ letter of a one-sided fixed point of a
constant-length substitution, then, for any positive integer
$b\geqslant 2$, the number $\sum_n f(n) \ts b^{-n}$ is either rational
or transcendental \cite{ABannals,AFplms,BBCplms}. Also, this number
cannot be a Liouville number, that is, it has finite irrationality
exponent \cite{ACcompo,BBCplms}. The partial sums
$S(N):=\sum_{n\leqslant N}f(n)$ satisfy even stronger gap
properties. If $S(N)$ is unbounded, there is a constant $c>0$ such
that $|S(N)|\geqslant c\, \log(N)$ for infinitely many integers $N$;
see \cite{BCH1,BCH2}. Viewing Lehmer's problem for Borwein polynomials
in this context may, at the least, take away some of the surprise of
its conclusion, and provide an additional reason to believe in the
conjecture for this class of polynomials. \smallskip

The remainder of this paper is organised as follows. In
Section~\ref{sec:subs}, we give details regarding binary substitutions
and their associated Fourier matrices, while we give the relevant
definitions on Lyapunov exponents in Section~\ref{sec:Lya}.
Section~\ref{sec:res} contains the proof of Theorem~\ref{main} as a
consequence of a more detailed version. In Section~\ref{sec:egs}, we
provide several examples, including those related to Littlewood,
Newman, and Borwein polynomials. Finally, in Section
\ref{sec:exandout}, we explore extensions to higher dimensions and
their relationship to logarithmic Mahler measures of multi-variable
polynomials.

\section{Substitutions and their Fourier matrices}\label{sec:subs}

As stated in the Introduction, we are concerned with binary constant-length
substitutions $\varrho$ defined on $\vS_2:= \{ 0,1 \}$ by
\begin{equation}\label{eq:def-rho}
   \varrho: \, \begin{cases} 0 \mapsto w^{}_{0} \\ 
     1\mapsto w^{}_{1}\, , \end{cases}
\end{equation}
where $w^{}_{0} \text{ and } w^{}_{1}$ are finite words over $\vS_{2}$
of equal length\footnote{Those comfortable with the dynamical setting
  will note that, by working with two prototiles of unit length, the
  tiling and symbolic pictures of these systems are equivalent
  (topologically conjugate by a sliding block map).}
$\lvert w^{}_{0}\rvert = \lvert w^{}_{1}\rvert = L\geqslant 2$.

We denote the $m^{\text{th}}$ column of $\varrho$ by
\[
   \cC_m \ts := \, \left[\begin{matrix} 
   \, (w_0 )^{}_m\, \\ \, (w_1 )^{}_m\,
   \end{matrix}\right] ,
\]
where $(w_i)_m$ is the $m^{\text{th}}$ letter of the word $w_i$. We
follow the convention of indexing the columns starting with $0$;
compare \cite[Ch.~4]{TAO}.  A binary substitution is said to have a
\emph{coincidence} at position $m$, if the column at that specific
position is either $\begin{bsmallmatrix} 0\\ 0
\end{bsmallmatrix}$
or $\begin{bsmallmatrix} 1\\ 1 \end{bsmallmatrix}$. A binary
substitution is called \emph{bijective} if there are no
coincidences.

For $0\leqslant i,j \leqslant 1$, let $T_{ij}$ be the set of all
positions $m$ where the letter $i$ appears in $w_j$, and let
$T:=(T_{ij})_{0\leqslant i,j\leqslant 1}$ be the resulting
$2\!\times\! 2$-matrix. Note that the substitution matrix $M_\varrho$,
as defined above, satisfies
\[
   M_{\varrho}  \,  = \, 
  \bigl( \mathrm{card} (T_{ij}) \bigr)_{0\leqslant
  i,j\leqslant 1} \ts .
\]
Using $T$, we build a matrix of pure point measures
$\delta^{}_T := (\delta^{}_{T_{ij}})_{0\leqslant i,j \leqslant 1},$ where
$\delta^{}_{S} := \sum_{x\in S} \delta^{}_x$ with
$\delta^{}_{\varnothing}=0$. This gives rise to an analytic
matrix-valued function via
\[
   B(k) \, := \, \overline{\widehat{\delta^{}_T}(k)} \ts ,
\]
which we call the \emph{Fourier matrix} of $\varrho$; see
\cite{BFGR,BG15}.  Note that $B (0) = M_\varrho $.  The Fourier matrix
provides more information than $M_{\varrho}$; it encodes the column
positions of each letter in the corresponding words that contain them,
whereas the entries of $M_\varrho$ only count the letters $0$ and $1$
in $w^{}_{0}$ and $w^{}_{1}$, respectively.

The following two examples are paradigmatic for the two principal
situations among aperiodic, binary constant-length substitutions.

\begin{example}\label{thue}
Consider the Thue--Morse substitution, as given by
\[  
  \varrho^{\vphantom{I_p}}_{\mathrm{TM}}: \begin{cases} 
   0\mapsto 01 \ts  \\ 1\mapsto 10 \ts .\end{cases}
\]  
Here, one has
  $T^{}_{\mathrm{TM}}= \left(\begin{smallmatrix} \{ 0 \} & \{ 1 \} \\
      \{ 1 \} & \{ 0 \} \end{smallmatrix} \right)$, which gives
\begin{flalign*}
   &&\delta^{}_{T^{}_{\mathrm{TM}}} \ts  = \,
     \begin{pmatrix}
     \delta^{}_0 & \delta^{}_1\\
     \delta^{}_1 & \delta^{}_0
     \end{pmatrix} 
   \quad \text{and} \quad
    B^{}_{\mathrm{TM}}(k) \, = \, \begin{pmatrix}
    1 & \ee^{2\pi \ii k}\\
    \ee^{2\pi \ii k} & 1
    \end{pmatrix}. &&\mbox{\exend}
\end{flalign*}
\end{example}

\begin{example}\label{pd}
On the other hand, for the period doubling substitution,
\[
   \varrho^{}_{\text{pd}}:\begin{cases}
   0\mapsto 01 \ts  \\ 1\mapsto 00 \ts ,\end{cases}
\] 
the corresponding
matrices are $T^{}_{\mathrm{pd}} =
\left( \begin{smallmatrix} \{ 0 \} & \{ 0, 1 \} \\
    \{ 1 \} & \varnothing \end{smallmatrix} \right)$ together with
\begin{flalign*}
   &&\delta^{}_{T_{\mathrm{pd}}} \, = \, 
   \begin{pmatrix}
   \delta^{}_0 & \delta^{}_0+\delta^{}_1\\
   \delta^{}_1 &  0
   \end{pmatrix}
  \quad \text{and} \quad
   B^{}_{\mathrm{pd}}(k) \, = \, \begin{pmatrix}
    1 & 1+\ee^{2\pi \ii k}\\
    \ee^{2\pi \ii k} & 0
    \end{pmatrix}.&&\mbox{\exend}
\end{flalign*}
\end{example}

\section{Lyapunov exponents}\label{sec:Lya}

Using the ergodic transformation $k\mapsto Lk\bmod 1 $ defined on the
$1$-torus $\TT$, which is represented by $[0,1)$ with addition modulo
$1$ and equipped with Lebesgue measure, one can use the Fourier matrix
$B(k)$ to build the matrix cocycle
\[
   B^{(n)}(k) \, := \,
   B(k) B(Lk)\cdots B (L^{n-1}k ) \ts ,
\]
where the (dynamically unusual) extension to the right originates from
the underlying spectral problem of binary substitutions; see
Remark~\ref{rem:spec} below and \cite{BG15,BGM,BaGriMa}. Recall that
the integer $L\geqslant 2$ is the common length of the words $w^{}_0$
and $w^{}_1$ from the definition of the binary substitution
$\varrho$. We note further that the inverse cocycle
$( B^{(n)} (k) )^{-1}$ exists for almost every $k\in \RR$, because
$\det \bigl( B(k)\bigr) =0$ for at most a countable subset of $\RR$.
Due to the ergodicity of the transformation $k \mapsto Lk$ on $\TT$
relative to Lebesgue measure, Oseledec's multiplicative ergodic
theorem ensures the existence of the Lyapunov exponents and the
corresponding subspaces in which they represent the asymptotic
exponential growth rate of the vector norms; see \cite{BP}. More
precisely, if $v\in\CC^2$ is any (fixed) row vector, the values
\begin{equation}\label{eq:def-L}
    \chi^{B} (v,k) \, := \lim_{n\to\infty}
   \myfrac{1}{n} \log\| v B^{(n)}(k) \|
\end{equation}
exist for Lebesgue-almost every $k\in \RR$ and are constant on a set
of full measure. In fact, as a function of $v$, they take only
finitely many values, at most two in this case. These values do not
depend on the choice of the norm in \eqref{eq:def-L}.  Moreover, in
the non-degenerate case, there exists a filtration
\[
   \{ 0 \} \, =: \,  \cV^{}_{0} \, \subsetneq\, 
  \cV^{}_{1} \,  \subsetneq\,
  \cV^{}_{2} \, := \, \mathbb{C}^{2}
\]
such that $\chi^{B}_{1}$ is the corresponding exponent for all
$0 \ne v\in\cV_{1}$ and $\chi^{B}_{2}$ for all
$v\in\cV_{2}\setminus\cV_{1}$.  A vector $v$ from the Oseledec
subspace $\cV_{i+1} \setminus \cV_{i}$ satisfies the property that,
for almost every $k\in \RR$, the norm $\|v B^{(n)}(k) \|$ has
exponential growth factor $\ee^{n\chi^{B}_{i+1}}$ as $n\to \infty$.
In general, these subspaces depend on $k$, and the filtration
simplifies in the obvious way when
$\chi^{B}_{1} (k) = \chi^{B}_{2} (k)$.  We refer the reader to the
monographs \cite{BP} for a general overview and \cite{Viana} for a
more elaborate discussion of linear cocycles.

It is well known that there exist at most two distinct exponents for
$2$-dimensional cocycles, denoted by $\chi^{B}_{\max} (k)$ and
$\chi^{B}_{\min} (k)$.  For invertible cocycles, these exponents
afford $v$-independent representations as
\begin{align*}
  \chi^{B}_{\max} (k ) \, & := \,  
   \lim_{n\to\infty} \myfrac{1}{n} \log \ts 
   \bigl\| B^{(n)} (k) \bigr\|   \quad \text{and} \\
   \chi^{B}_{\min} (k ) \, & := \, 
   - \lim_{n\to\infty} \myfrac{1}{n}
   \log \ts \bigl\| \bigl(B^{(n)}  (k )\bigr)^{-1}\bigr\| .
\end{align*}
Moreover, the exponents are constant almost everywhere, hence
effectively independent of $k$ as well. This means that we are dealing
with two \emph{numbers}, $\chi^{B}_{\max}$ and $\chi^{B}_{\min}$. It
turns out that one has an unexpected connection with logarithmic
Mahler measures, as we discuss below.

\begin{remark}\label{rem:spec}
  These exponents come up in the spectral study of the substitutions
  associated to these Fourier matrices, and can be derived from a
  renormalisation scheme involving pair correlation functions; see
  \cite{BG15,BGM,BGr}.  In particular, by measure-theoretic arguments,
  one can conclude that, if
  $\lvert \chi^{B}_{\max}\rvert < \log \sqrt{L}$, the diffraction
  measure, for an arbitrary choice of weights, never has an absolutely
  continuous component.  We refer to the literature for further
  details, namely to \cite{Neil} for the binary constant length case,
  and to \cite{BFGR,BaGriMa} for an appropriate extension to a family
  of non-Pisot substitutions, via the corresponding inflation tiling.
  \exend
\end{remark}

\section{Proof of the main result}\label{sec:res}

As stated in Theorem~\ref{main} in the Introduction, the maximal
Lyapunov exponent can be written as a logarithmic Mahler measure. We
prove this as Theorem~\ref{thm:min-max} below, which is a more
detailed version of Theorem~\ref{main}; compare \cite{Neil}.

\begin{lemma}\label{lem:mat}
  Let\/ $\varrho$ be a substitution as specified in
  Eq.~\eqref{eq:def-rho}.  Consider the sets
\[
    P^{}_a := \big\{ m \mid \cC_m = 
  \begin{bsmallmatrix} 0\\ 1 \end{bsmallmatrix} \big\}
  \quad \text{and} \quad
  P^{}_b := \big\{ m \mid \cC_m = \begin{bsmallmatrix}
    1\\0 \end{bsmallmatrix} \big\} ,
\]
which collect bijective positions of the same type.  Further, let\/
$z=\ee^{2 \pi \ii k}$ and set
\[
   Q (z) \, := \, \overline{\widehat{\delta_{P_a}} (k)}
   \quad \text{and} \quad
   R (z) \, := \, \overline{\widehat{\delta_{P_b}} (k)}.
\]
Then, $\det \bigl( B(k) \bigr) = p^{}_L(z)\cdot \bigl(
Q-R \bigr)(z)$,
  where\/  $p^{}_{L} (z) = 1+z+\cdots + z^{L-1}$.
\end{lemma}

\begin{proof} 
  Similar to the definitions of $Q$ and $R$ above, define the sets
  $P^{}_0 := \left\{ m \mid \cC_m =
  \begin{bsmallmatrix} 0\\ 0 \end{bsmallmatrix} \right\}$ and
  $P^{}_1 := \left\{ m \mid \cC_m = \begin{bsmallmatrix}
    1\\1 \end{bsmallmatrix} \right\},$
  and let
\[
   S_0 (z) \, := \, \overline{\widehat{\delta_{P_0}} (k)}
   \quad \text{and} \quad
   S_1 (z) \, := \, \overline{\widehat{\delta_{P_1}} (k)}.
\]
In general, the Fourier matrix of $\varrho$ satisfies
\[
   B(k) \, = \, \begin{pmatrix}
   \bigl( S^{}_{0}+Q \bigr) (z) & 
   \bigl( S^{}_{0}+R \bigr) (z) \\ 
   \bigl( S^{}_{1}+R \bigr) (z) & 
   \bigl( S^{}_{1}+Q \bigr) (z)
    \end{pmatrix} 
    \quad \text{with} \, z = \ee^{2 \pi \ii k} .
\]
Since there are only four distinct column types, we see that
\[
	S^{}_0+S^{}_1+Q+R \, = \, 
   p^{\vphantom{I_p}}_{L} \ts .
\] 
One can now verify the lemma by direct computation.
\end{proof}

Recall that, as a consequence of Oseledec's multiplicative ergodic
theorem, our Lyapunov exponents exist, and are constant, for almost
every $k\in\RR$. We call them $ \chi^{B}_{\min}$ and
$ \chi^{B}_{\max}$.

\begin{theorem}\label{thm:min-max}
  For any primitive, binary constant-length substitution\/ $\varrho$,
  the extremal Lyapunov exponents are explicitly given by
\[
  \chi^{B}_{\min}  \, = \, 0 
  \quad \text{and} \quad
  \chi^{B}_{\max}  \, = \, \fm  (Q-R) \ts ,
\]
   with\/ $Q$ and\/ $R$ as in Lemma~$\ref{lem:mat}$.
\end{theorem}  

\begin{proof}
  Aside from the existence of the extremal Lyapunov exponents as
  limits, Oseledec's multiplicative ergodic theorem \cite{BP,Viana}
  also guarantees forward Lyapunov regularity almost everywhere. 
  That is, for almost every $k\in\RR$, the sum of the exponents is 
  given by
\begin{equation}\label{eq:sum}
  \chi^{B}_{\min} (k) + \chi^{B}_{\max} (k) \, =
  \lim_{n\to\infty} \myfrac{1}{n} \log \ts
  \bigl| \det \bigl( B^{(n)} (k) \bigr) \nts \bigr| ,
\end{equation}
where one can argue that, for the matrices above, the limit on the
right-hand side converges for almost every $k\in\RR$ to
\[
   \fm ( Q-R ) \, = \int^{1}_{0}\log \ts
   \bigl| (Q-R) \bigl(\ee^{2 \pi \ii t} \bigr)
    \bigr| \dd t \ts .
\]
This can be seen by an application of Birkhoff's ergodic theorem,
because $t \mapsto L t $ on $\TT$ is ergodic for Lebesgue measure, and
$t \mapsto (p^{}_{L} \! \cdot
(Q-R ) ) \bigl(\ee^{2 \pi \ii t}\bigr)$ defines a function in
$L^{1} (\TT)$.  The claim then follows from the multiplicative
property of the determinant in conjunction with Lemma~\ref{lem:mat}
and the fact that $\fm (p^{}_{L}) = 0$. This value follows via Jensen's
formula because all zeros of $p^{}_{L}$ are roots of unity. 

Next, we note that the row vector $(1,1)$ is a left eigenvector of
$B(k)$, for all $k\in\RR$, with eigenvalue
$p^{}_{L} \bigl(\ee^{2 \pi \ii k}\bigr)$. Hence, using this specific
direction, we get one of the exponents to be
$\chi^{B}_{1} = \fm (p^{}_{L}) = 0$.   From the
sum in Eq.~\eqref{eq:sum}, and from the fact that the logarithmic
Mahler measure of an integer polynomial is always non-negative, we
then get that the exponent corresponding to this invariant subspace is
the minimal one, $\chi^{B}_{1} = \chi^{B}_{\min}$, the other being
$\chi^{B}_{\max} = \fm (Q-R)$.
\end{proof}

Note that the result of this theorem is not restricted to
bijective substitutions, even though only the bijective
positions matter for the exponents.

\begin{remark}\label{rem:more}
  In the proof of Theorem~\ref{thm:min-max}, instead of invoking
  Birkhoff's ergodic theorem, one can also work with the uniform
  distribution of $(L^n k)^{}_{n\in\NN}$ modulo $1$ for almost every
  $k\in\RR$ and Weyl's lemma. The difficulty to overcome here is that
  the function defined by
  $k \mapsto \log \ts \ts \bigl| \det \bigl( B (k) \bigr) \bigr|$
  generally has singularities.  Fortunately, they are isolated (hence
  at most countable), and one can extend Weyl's result for locally
  Riemann integrable $1$-periodic functions to this case via Sobol's
  theorem in conjunction with Diophantine approximation and
  discrepancy analysis; see \cite{BHL} and references therein for a
  more comprehensive discussion.
  
  It is also interesting to observe that one can obtain
  $\chi^{B}_{\max}$ via the ($k$-independent) right eigenvector
  $\tilde{v}=\left(\begin{smallmatrix} 1\\
      -1 \end{smallmatrix}\right)$
  of $B (k)$, with corresponding eigenvalue
  $( Q-R ) \bigl(\ee^{2 \pi \ii k}\bigr)$ as before. One then has, for
  almost every $k\in\RR$, that
\[
    \lim_{n\to\infty} \myfrac{1}{n} \log \bigl\|
    B^{(n)} (k) \ts \tilde{v} \bigr\| \, =  \lim_{n\to\infty}
    \myfrac{1}{n} \Bigl( \| \tilde{v} \| \, + 
    \sum_{\ell=0}^{n-1} \log \ts \bigl|
    ( Q-R ) \bigl( \ee^{2 \pi \ii L^\ell k} \bigr) \bigr| \Bigr)
    \, = \, \fm (Q-R) \ts ,
\]
where the last step once again relies on Birkhoff's
ergodic theorem or on the remarks of the preceding paragraph.
\exend
\end{remark}

In the general setting of Theorem~\ref{thm:min-max}, one gets a
stronger result assuming periodicity.  We require the following lemma,
where we use the common shorthand $\varrho = (w^{}_{0} , w^{}_{1})$
for the substitution from Eq.~\eqref{eq:def-rho}.

\begin{lemma}\label{lem:subst-per}
  Let\/ $\varrho$ be a primitive, binary constant-length substitution
  that defines a periodic hull. Then, one either has\/
  $\varrho = (w,w)$ with\/ $w$ containing at least one copy each of
  the letters\/ $a$ and\/ $b$, or\/ $\varrho$ is bijective, and of the
  form\/ $\varrho = \bigl( (ab)^m a, (ba)^m b \bigr)$ or\/
  $\varrho = \bigl( (ba)^m b, (ab)^m a\bigr)$ for some\/ $m\in\NN$.
\end{lemma}

\begin{proof}
  Clearly, any substitution $\varrho = (w,w)$ with
  $\lvert w \rvert >1$ defines a periodic hull, and primitivity
  implies that $w$ must contain both letters. Consequently, we can now
  focus on $\varrho = (w^{}_{0} , w^{}_{1})$ with
  $w^{}_{0} \ne w^{}_{1}$. Let us begin with the cases of equal
  prefix.

  If $\varrho = (r u s, r v s)$, where
  $\lvert u \rvert = \lvert v \rvert$ and arbitrary
  $r,s \in \vS_2$, we may consider
  $\varrho^{\ts\prime} = (s \ts r u, s \ts r v)$ instead,
  because $\varrho$ and $\varrho^{\ts \prime}$ are conjugate and thus
  define the same hull \cite[Prop.~4.6]{TAO}. Since we only
  consider the case $w^{}_{0} \ne w^{}_{1}$, we must have at least one
  position where they differ, and we may assume that this happens at
  the last position.

  For $\varrho = (a u a, avb)$, the words $ab$ and $ba$ are both legal
  (as $u$ must contain the letter $b$ by primitivity), and an
  iteration of the corresponding seeds under $\varrho$ results in the
  sequences
 \[
\begin{array}{ccccccc}
    a|b  & \mapsto & \ldots a|a \ldots &  \mapsto &
    \ldots a|a \ldots & \mapsto & \cdots  \\
    \ts b\ts |a  & \mapsto & \ldots b\ts |a \ldots & \mapsto &
    \ldots b\ts |a \ldots & \mapsto & \cdots
\end{array}
\]
that converge to two-sided fixed points. Since they are proximal
(equal to the right, but not to the left) by construction, the hull of
$\varrho$ must be aperiodic \cite[Cor.~4.2]{TAO}. A completely
analogous argument works for $\varrho = (bua, bvb)$, which is again
aperiodic.

Likewise, for $\varrho = (a u b, a v a)$, the word $ab$ is legal,
hence also $ba$. Using the latter as seed, we get the iteration
\[
   b|a \; \mapsto \; \ldots a|a \ldots 
   \; \mapsto \; \ldots b|a \ldots \; \mapsto \; \cdots
\]
that ultimately alternates between two elements that form a proximal
pair, which implies aperiodicity. Analogously, for
$\varrho = (bub, bva)$, we get a proximal pair (and hence
aperiodicity) from an iteration that starts with the legal seed $b|a$,
which is mapped to $a|b$ and then alternates between $b|b$ and $a|b$.

Consequently, periodic cases for $w^{}_{0} \ne w^{}_{1}$ can only
occur if the two words have unequal prefix \emph{and} unequal suffix. When
$\varrho = (aub, bva)$, the seed $b|a$ is legal, which under iteration
alternates between $a|a$ and $b|a$; when $\varrho = (bua, avb)$, one
has the matching situation with $a|b$ and $a|a$, so both cases possess
proximal pairs and are thus aperiodic.

Finally, if $\varrho= (aua, bvb)$, one gets a proximal pair if and
only if $aa$ or $bb$ is legal, and the same statement applies to
$\varrho^{\ts \prime} = (bub, ava)$. The only way this can be avoided
is if $w^{}_{0}$ and $w^{}_{1}$ both alternate between $a$ and $b$,
which indeed gives periodic hulls, and these substitutions are the two
other cases stated.
\end{proof}

\begin{theorem}\label{thm:per0} 
  If the primitive, binary constant-length substitution\/ $\varrho$
  defines a periodic hull, the extremal Lyapunov exponents satisfy\/
  $\, \chi^{B}_{\min} = \chi^{B}_{\max} = 0$.
\end{theorem}

\begin{proof} 
  In view of Lemma~\ref{lem:subst-per}, we have to check the claim for
  three cases.  When $\varrho = (w,w)$, where $w$ contains both
  letters and $L = \lvert w \rvert \geqslant 2$, we consider an
  arbitrary starting vector $v = (\alpha, \beta) \ne 0$.  For $n>1$,
  one then has
\[
    v B^{(n)} (k) \, = \,
   \bigl( \alpha \ts S^{}_{0} \bigl(\ee^{2 \pi \ii k}\bigr) +
   \beta \ts S^{}_{1} \bigl(\ee^{2 \pi \ii k}\bigr) \bigr)
   \cdot (1,1) \ts B^{(n-1)} (L k) \ts ,
\]
where $(1,1)$ is a left eigenvector of $B^{(n-1)} (L k)$. Since
$S^{}_{0} + S^{}_{1} = p^{}_{L}$ in this case, one finds
\[
    \| v B^{(n)} (k) \| \, = \, \bigl| \alpha \ts S^{}_{0}
    \bigl( \ee^{2 \pi \ii k} \bigr)
    + \beta \ts S^{}_{1} \bigl( \ee^{2 \pi \ii k} \bigr)
    \bigr| \, \| (1,1) \|  \prod_{\ell=1}^{n-1}
   \bigl| p^{}_{L} \bigl( \ee^{2 \pi \ii L^\ell k} \bigr) \bigr| 
   \ts .
\]
The first term on the right-hand side only vanishes at isolated (and
hence countably many) values of $k$, which we may exclude. Then, a
calculation with Birkhoff averages shows that, for almost every
$k\in\RR$, we get
\[
   \lim_{n\to\infty} \myfrac{1}{n} \log  
   \| v B^{(n)} (k) \| \, = \, \fm (p^{}_{L} ) \, = \, 0 \ts ,
\]
which establishes the claim in this case.

If $\varrho$ is bijective, we have $L=2m+1$ for the two remaining
cases by Lemma~\ref{lem:subst-per}. In line with our previous
reasoning, the Fourier matrix is of the form
\[
    B (k) \, = \, \begin{pmatrix}
    Q(z) & R(z) \\ R(z) & Q(z) \end{pmatrix} 
    \quad \text{with} \, z = \ee^{2 \pi \ii k} ,
\]
where the polynomials $Q$ and $R$ satisfy
$Q(z) + R(z) = p^{}_{L} (z) = 1 + z + \cdots + z^{2m}$, which is
cyclotomic, so that $\fm (Q+R) = 0$. Also, due to the alternating
structure of $\varrho (a)$ and $\varrho(b)$, one has
$\bigl( Q-R \bigr) (z) = \pm \bigl( Q+R \bigr) (-z)$. This means that
$Q-R$ is cyclotomic as well, and $\fm (Q-R) = 0$. Now, one sees that
\[
    (1, \pm 1) B^{(n)}(k) \, = \, (1, \pm 1)
    \prod_{\ell=0}^{L-1} ( Q \pm R ) 
    \bigl(\ee^{2 \pi \ii  L^\ell k} \bigr) 
\]
which, for almost every $k\in\RR$, gives the two exponents as
$\fm (Q+R) = 0$ and $\fm (Q-R) = 0$ by a simple calculation as in
Remark~\ref{rem:more}.  This implies our claim for these two cases.
\end{proof}

The converse of Theorem \ref{thm:per0} does not hold. For example,
both the Thue--Morse and the period doubling substitutions, given in
Examples~\ref{thue} and \ref{pd}, have
$\chi^{B}_{\min} = \chi^{B}_{\max}=0$; this means that the norm of the
resulting vector after applying their respective cocycles to any
starting vector $v$ does not grow exponentially. However, one must be
careful here, as zero Lyapunov exponents do not exclude
sub-exponential growth behaviour.

\section{Examples: From polynomials to substitutions}\label{sec:egs}

Theorem~\ref{main}, and the more specific Theorem~\ref{thm:min-max},
allow one, given a binary constant-length substitution, to write down
a polynomial whose logarithmic Mahler measure determines the maximal
Lyapunov exponent related to the substitution. But the nature of our
results allows one to go the other way as well, as we now do. We
explain how, given a specific polynomial, one can build a substitution
associated to the maximal Lyapunov exponent for the cocyle
$B^{(n)}(k)$. We also comment on the essential uniqueness of these
substitutions and the properties of their Fourier matrices. We focus
on specific classes of height-$1$ polynomials that have been important
in the study of the logarithmic Mahler measure in the context of
Lehmer's problem.

\begin{example}[Littlewood polynomials]\label{ex:Little}
  Recall that a polynomial $q(z)=\sum_{m=0}^{n-1} c^{}_m \ts z^m$ of
  degree $n-1$ with coefficients $c_m\in \{ -1,1 \}$ is called a
  \emph{Littlewood polynomial of order\/ $n-1$}; see \cite{BorJ,BorM,
    Moss2003}. As before, let $\cC_m$ be the $m^{\text{th}}$
  column of $\varrho$.  Starting with the polynomial, we choose
  $\cC_m$ to be
\[
   \cC_m \, = \, \begin{cases}
   \mbox{\raisebox{1pt}{$\begin{bsmallmatrix}   0 \\ 1
   \end{bsmallmatrix}$}},  & \text{if } c^{}_m=1 , \\
   \mbox{\raisebox{1pt}{$\begin{bsmallmatrix}   1 \\ 0
   \end{bsmallmatrix}$}},  & \text{if } c^{}_m=-1 ,
   \end{cases}
\]
and we build the substituted words for $0$ and $1$ by looking at the
concatenation $\cC^{}_0 \, \cC^{}_1 \! \cdots \cC^{}_{L-1}$. Since
there are only two possible column types, we see immediately that the
sets $P^{}_a$ and $P^{}_b$ from Lemma~\ref{lem:mat} satisfy
\[
    P^{}_a \cup P^{}_b    \, = \,  \{0,1,\ldots , L-1\}\, ,
\]
and also that the resulting
substitution $\varrho$ must be bijective.  By construction, we have
\begin{equation}\label{eq:tmp}
   \chi^{B}_{\max}\, =\, \fm (Q-R)\, = \, \fm (q)
\end{equation}
for the cocycle defined by the Fourier matrix associated to $\varrho$,
which explicitly reads
\[
   B (k) \, = \, \begin{pmatrix}
       Q (z) & R (z) \\ R (z) & Q (z) \end{pmatrix}
        \quad \text{where } z=\ee^{2 \pi \ii k}.
\]
Note that, in this case, $\chi^{B}_{\max}$ can also be calculated by
observing that $(1,-1)$ is a $k$-independent left eigenvector of
$B (k)$ with eigenvalue $( Q - R ) \bigl( \ee^{2 \pi \ii k} \bigr)$,
thus also giving \eqref{eq:tmp}.

The substitution corresponding to $q = Q - R$ is
essentially unique, up to the obvious freedom that emerges from the
relation $\fm (-q) = \fm (q)$, that is, from changing all signs. This
is the case because a given sequence of signs uniquely
specifies the columns of $\varrho$.  For example, let us consider the
polynomial $q(z)=-1-z+z^2-z^3+z^4$, where we get the substitutions
\[
   \varrho_{q} :\begin{cases}
   0\mapsto 11010\\
   1\mapsto 00101
   \end{cases}
   \quad \text{and} \quad \;\;
   \varrho_{-q} :\begin{cases}
   0\mapsto 00101 \\
   1\mapsto 11010
   \end{cases}
\]
with asssociated Fourier matrices
\[
   B_{q} (k) \, = \, \begin{pmatrix}
   \ee^{4\pi \ii k}+\ee^{8\pi \ii k} & 
   1+\ee^{2 \pi \ii k}+\ee^{6 \pi \ii k}\\
   1+\ee^{2 \pi \ii k}+\ee^{6 \pi \ii k} & 
   \ee^{4\pi \ii k}+\ee^{8\pi \ii k} 
   \end{pmatrix} 
   \quad \text{and} \quad
   B_{-q} (k) \, = \, \begin{pmatrix}
   0 & 1 \\ 1 & 0 \end{pmatrix} B_{q} (k) \ts .
\]
Both induce a cocycle whose maximum Lyapunov exponent
is \begin{flalign*} &&
  \chi^{B}_{\max} \, = \, \fm (q)
  \, \approx \, 0.656 {\ts} 256 \ts .&&
  \mbox{\exend}
\end{flalign*}
\end{example}

The Fourier matrices associated to bijective substitutions enjoy
further properties such as simultaneous diagonalisibility and a
$k$-independent expression for the Oseledec splitting.  In these
cases, one always has $\cV_1= \CC \ts (1,1)$.  This means that, for
all vectors $v \ne 0$ in this subspace, the asymptotic exponential
growth rate is $0$, for almost every $k \in \RR$.  One also sees that
every column sum and every row sum of $B(k)$ is the cyclotomic
polynomial $p^{}_L $; for rows it is due to the bijectivity of the
substitution, for columns it follows from the constant-length
property. 

Before we continue, let us mention that adding coincident positions to
a given constant-length substitution as prefix or suffix does not
change the Lyapunov exponents. Conversely, any primitive, binary
constant-length substitution either starts and ends with bijective
positions, or can be conjugated into a substitution that either has
coincident prefix or suffix positions, but not both. The period
doubling case is an example of this. To avoid pathologies, we now
restrict our attention to substitutions which do not end with a
coincidence.

\begin{example}[Newman polynomials]
  For $\,\{ 0,1 \}$-polynomials with constant term~$1$,
  also known as \emph{Newman polynomials} \cite{Moss2003}, one has
  $R = 0$, so the associated Fourier matrix is
\[
     B(k) \, = \, \begin{pmatrix} \bigl( 
    S^{}_0 + Q \bigr) (z) & S^{}_{0}(z) \\
      S^{}_1(z)  & \bigl(S^{}_1 +  Q \bigr) (z) 
   \end{pmatrix} \quad \text{with } \, z = \ee^{2 \pi \ii k} ,
\]
which leads to $\chi^{B}_{\max} = \fm (Q)$ by
Theorem~\ref{thm:min-max}.  If either $S^{}_{0}$ or $S^{}_{1}$ is
zero, $M_\varrho = B(0)$ is a triangular matrix, and cannot be
primitive. This can be avoided if at least two coefficients of the
polynomial are zero. If there is only one, we can still construct a
primitive substitution by recalling that $\fm (-q) = \fm (q)$, so we
only need to exchange the two bijective column types.
 
As a concrete example, consider $q(z) = 1 + z^2$. The two standard choices
\[
   \varrho^{}_q:\begin{cases} 0 \mapsto 000\\ 
   1 \mapsto 101\end{cases}\quad \text{and}
   \qquad \varrho^{}_{q'}:\begin{cases} 0 \mapsto 010\\ 
  1 \mapsto 111\end{cases}
\] 
both give non-primitive substitutions; in fact, their substitution
matrices are not even irreducible. However,
\[
   \varrho^{}_{-q}:\begin{cases} 0 \mapsto 101\\ 
   1 \mapsto 000\end{cases}\quad\text{and}
   \qquad \varrho^{}_{-q'}:\begin{cases}0 \mapsto 111\\ 
   1 \mapsto 010\end{cases}
\]
are both primitive and aperiodic, and have $\chi^{B}_{\max} = \fm (q)$.
One can see in this example that replacing $q$ by $-q$ really
means an exchange of $w^{}_{0}$ and $w^{}_{1}$ in the definition of
$\varrho$.

As another example, consider the reciprocal Newman polynomial
\[
    q(z) \, = \, 1 + z^3 + z^4 + z^5+z^6+z^7+z^8+z^9 +
      z^{10} + z^{11} + z^{14}
\]
taken from \cite[p.~1375]{Boyd-1}. One choice for a substitution
(with $L=15$) is
\[
   \varrho^{}_{q}  :  \begin{cases}
     0\mapsto 010000000000000 \\
     1\mapsto 110111111111001
    \end{cases}
\]
which means $S^{}_1 (z) = z$ and $S^{}_0 (z) = z^2 + z^{12} + z^{13}$,
together with $Q = q $. Here, one has
$\fm (q) \approx \log (1.265 {\ts} 122 )$.  This value is strictly
smaller than $\log ( \lambda_{\mathrm{p}} )$, where
$\lambda_{\mathrm{p}} \approx 1.324 {\ts} 718$ is the \emph{plastic
  number} as described in the Introduction.
Recall that $\log ( \lambda_{\mathrm{p}} )$ is the
sharp lower bound for $\fm (p)$ over all non-reciprocal
integer polynomials $p$ that are not a product of a monomial with
a cyclotomic polynomial. \exend
\end{example}

Note that, when associating a polynomial to a binary constant-length
substitution $\varrho$, it is only the bijective columns of $\varrho$
that are determined by the non-zero coefficients. Thus, we can extend
the construction to generic height-$1$ polynomials, even when the
constant term is zero, as in the period doubling example. However,
when interested in non-trivial logarithmic Mahler measures, one can
assume that the constant coefficient is non-zero.

\begin{example}[Borwein polynomials] 
  When considering Borwein polynomials, with non-zero constant
  coefficients, one can choose the columns of an associated
  substitution just as in Example~\ref{ex:Little}, but with the
  additional freedom to vary the choice for each zero coefficient. As
  in Example~\ref{ex:Little}, starting with the polynomial, we choose
  $\cC_m$ to be
\[
   \cC_m \, = \, \begin{cases}
   \mbox{\raisebox{1pt}{$\begin{bsmallmatrix}   0 \\ 1
   \end{bsmallmatrix}$}},  & \text{if } c^{}_m=1 , \\
   \mbox{\raisebox{1pt}{$\begin{bsmallmatrix}   1 \\ 0
   \end{bsmallmatrix}$}},  & \text{if } c^{}_m=-1 ,\\
   \mbox{\raisebox{1pt}{$\begin{bsmallmatrix}   0 \\ 0
   \end{bsmallmatrix}$}} \mbox{ or } 
   \mbox{\raisebox{1pt}{$\begin{bsmallmatrix}   1 \\ 1
   \end{bsmallmatrix}$}},  & \text{if } c^{}_m=0 ,
   \end{cases}
\]
and we build the substituted words for $0$ and $1$ by looking at the
concatenation $\cC^{}_0 \, \cC^{}_1 \! \cdots \cC^{}_{L-1}$.  There
are two choices for each zero coefficient of the polynomials, so that,
if $p$ is a Borwein polynomial of degree $L-1$ with $n$ zero
coefficients, there are $2^n$ distinct binary constant-length
substitutions of length $L$ whose maximal Lyapunov exponents are all
given by $\fm (p)$. On top of this freedom, we can also still work
both with $p$ or with $-p$, as we saw earlier.

As a concrete example, we consider Lehmer's polynomial
$\ell^{}_{\mathrm{L}}$ from \eqref{eq:def-L}, where
$c^{}_2 = c^{}_8 = 0$. Recall from the Introduction that this
polynomial is irreducible, and has precisely one root outside the unit
disk, which is real and Salem.  Recall further that
$\ell^{}_{\mathrm{L}}$ is the polynomial with the smallest known
positive logarithmic Mahler measure,
$\fm (\ell^{}_{\mathrm{L}}) \approx \log (1.176 {\ts} 281 )$.  Here,
\[
     \varrho^{}_{\ell^{}_{\mathrm{L}}}  :  \begin{cases}
     0\mapsto 00111111000 \\
     1\mapsto 11100000011
    \end{cases}
\]
is one of the eight substitutions that correspond 
to the polynomial $\ell^{}_{\mathrm{L}}$. 
\exend
\end{example}

With this connection and representation, we obtain the following
equivalent reformulation of the strong version of Lehmer's problem for
Borwein polynomials.

\begin{question}
  Does there exist a primitive, binary constant-length substitution
  $\varrho$ with maximal Lyapunov exponent
  $0<\chi^{B}_{\max} < \fm (\ell^{}_{\mathrm{L}}) \approx \log (1.176
  {\ts} 280 {\ts} 818) \,$?
\end{question}

\section{Extensions and outlook}\label{sec:exandout}

Lyapunov exponents are neither restricted to constant-length
substitutions nor to binary alphabets. In fact, there are many
extensions possible; see \cite{BFGR,BGM} and references therein for
more. In general, however, the Lyapunov exponents are no longer
logarithmic Mahler measures themselves, though they often still
satisfy interesting estimates in such a setting.

Moreover, there is actually also a generalisation to higher
dimensions, as briefly stated in \cite{Neil-MFO}. Here, one considers
stone inflation rules of finite local complexity \cite{TAO} and
selects a suitable marker (or reference point) in each prototile (such
that the tiling and the point set are mutually locally derivable 
\cite[Sec.~5.2]{TAO} from each other). One particular class emerges 
from \emph{block substitutions}, such as those discussed in 
\cite{squiral,Natalie}.

\begin{example}
A simple bijective example is given by
\[
    a \, \mapsto \, \mbox{\Large $\begin{smallmatrix}
    b & a \\ a & b \end{smallmatrix} $} \; , \quad
    b \, \mapsto \, \mbox{\Large $ \begin{smallmatrix}
    a & b \\ b & a \end{smallmatrix} $} 
\]
which is primitive and aperiodic.  Here, one can represent the two
letters by unit squares with a (coloured) reference point at their
lower left corners. Then, with $\bs{k} = (k_1, k_2) \in \RR^2 $, one
finds the Fourier matrix
\[
     B (\bs{k}) \, = \, \begin{pmatrix}
     1 +  x \ts y  &  x+y \\ x + y & 1 + x \ts y
     \end{pmatrix}
     \quad \text{where } \, (x,y) =
     \bigl( \ee^{2 \pi \ii k_1} ,  \ee^{2 \pi \ii k_2} \bigr) ,
\]
which satisfies $(1, \pm 1) B(\bs{k}) = \bigl( (1+xy) \pm (x+y)\bigr)
(1, \pm 1)$. Since $1+x+y+x \ts y =(1+x)(1+y)$ and
$1-x-y+x \ts y = (1-x)(1-y)$, all factors are cyclotomic.
Consequently, for almost every $\bs{k}\in\RR^2$, one gets the
Lyapunov exponents as
\[
      \chi^{B}_{\min}  \, = \,   \fm (1 + x + y + x \ts y) = 0
      \quad \text{and} \quad
      \chi^{B}_{\max} \, = \, \fm (1-x-y + x\ts y) = 0 \ts ,
\]     
which resembles Example~\ref{thue} in many ways. Here, in line with
Eq.~\eqref{eq:def-M}, the logarithmic Mahler measure of a
multi-variable polynomial $p$ is defined as
\[
    \fm (p) \, = \int_{[0,1]^d} \log \, \bigl| 
    \ts p \bigl( \ee^{2 \pi \ii t_1}, \ldots ,
    \ee^{2 \pi \ii t_d} \bigr) \bigr| \dd t^{}_{1} \cdots
    \dd t^{}_{d} \ts .
\]

As another example, consider the bijective block substitution
\[
    a \, \mapsto \, \mbox{\Large $\begin{smallmatrix}
    b & a & b \\ a & a & a \\ b & a & b \end{smallmatrix} $} \; , \quad
    b \, \mapsto \, \mbox{\Large $ \begin{smallmatrix}
    a & b & a \\ b & b & b \\ a & b & a \end{smallmatrix} $} 
\]
that emerges from the \emph{squiral tiling} \cite{squiral}. Here,
one has $\chi^{B}_{\min} = \fm \bigl( (1+x+x^2)(1+y+y^2) \bigr) = 0$
as before, while
\[
    \chi^{B}_{\max} \,  = \,  \fm \bigl(x + y (1+x+x^2) + x \ts y^2 - 
    (1 + x^2)(1+ y^2) \bigr)  \, \approx \, 0.723 {\ts} 909
\]
is strictly positive. \exend
\end{example}

It is clear that one can now repeat a lot of our previous analysis for
the class of bijective block substitutions, in any dimension. As is
implicit in \cite{Neil-MFO}, the blocks need not be cubes, as long as
they have length at least $2$ in each direction. We leave further
experimentation along these lines to the interested reader. Outside the
bijective class, interesting new phenomena are possible as follows.

\begin{example}\label{ex:1xy}
Consider the binary block substitution
\[
    a \, \mapsto \, \mbox{\Large $\begin{smallmatrix}
    b & a \\ b & b \end{smallmatrix} $} \; , \quad
    b \, \mapsto \, \mbox{\Large $ \begin{smallmatrix}
    a & a \\ a & a \end{smallmatrix} $} 
\]
which is clearly primitive. It has a coincidence, so that the
higher-dimensional analogue of Dekking's result, see 
\cite{squiral,Bart,Natalie},  implies the pure point spectral 
nature of the corresponding dynamical system (under the 
action of the $\ZZ^2$-shift).
    
Here, the Fourier matrix reads
\[
     B (\bs{k}) \, = \, \begin{pmatrix}
     x \ts y  &  (1+x)(1+y ) \\  1 + x + y & 0
     \end{pmatrix} ,
    \quad \text{where } \, (x,y) =
     \bigl( \ee^{2 \pi \ii k_1} ,  \ee^{2 \pi \ii k_2} \bigr) 
\]
and $\det \bigl( B(\bs{k}) \bigr) = - (1+x)(1+y)(1+x+y)$. As before,
we get $\chi^{B}_{\min} = \fm \bigl( (1+x)(1+y) \bigr) = 0$, while the
sum then satisfies
\[
\begin{split}
    \chi^{B}_{\min} + \chi^{B}_{\max} \, & = \, 
     \chi^{B}_{\max} \, = \, \fm
     \bigl( (1+x)(1+y)(1+x+y) \bigr) \, = \,  \fm ( 1+x+y) \\[1mm]
     & = \, \frac{3 \sqrt{3}}{4 \ts \pi} \,
     L (2,\chi^{}_{-3}) \, =\, L' (-1,\chi^{}_{-3}) 
     \, =  \, 2 \int_{0}^{1/3} \log
     \bigl( 2 \cos (\pi t) \bigr) \dd t
     \, \approx \,  0.323 {\ts} 066 \ts ,
\end{split}
\]
with $L(s,\chi^{}_{-3})$ denoting the Dirichlet $L$-function of the
character $\chi^{}_{-3} (n) = \bigl( \frac{-3}{n}\bigr)$, written in
terms of the Legendre symbol; compare \cite{BoydM}.  This special
value of a Mahler measure also appears in \cite{Smyth81,Smyth}, and
various other relations of this kind are known \cite{BL}, such as
\[  \fm (1+x+y+z) \, = \, \myfrac{7}{2 \ts \pi^2} \, 
   \zeta (3) \ts .
\]
The latter can clearly be realised by a block substitution in three
dimensions.  \exend
\end{example}

\begin{remark}\label{rem:Wannier}
  It seems hardly known that the Mahler measure from
  Example~\ref{ex:1xy}, together with its integral representation,
  first appeared in Wannier's calculation of the groundstate entropy
  of the antiferromagnetic Ising model on the triangular lattice
  \cite{Wannier}. This might be due to the fact that the numerical
  value originally given by him was incorrect, though corrected in an
  erratum $23$ years later. Somewhat similar entropy calculations
  in terms of logarithmic integrals later appeared in various other
  papers on solvable models of statistical physics. 

To be more precise, Wannier gets the entropy $s$ as a
double integral from which we obtain
\[
\begin{split}
   s \, & = \, \myfrac{1}{2} \int_{0}^{1} \int_{0}^{1}
   \log \bigl( 1 + 4\ts \cos(2 \pi v)^2 - 4 \ts
   \cos(2 \pi u) \cos (2 \pi v) \bigr) \dd u \dd v \\[2mm]
  & = \, \myfrac{1}{2}  \int_{0}^{1} \int_{0}^{1}
   \log \big| (1 + y^2 - x \ts y ) (x + x \ts y^2 - y) 
   \big|^{}_{x = \ee^{2 \pi \ii u}, \, y = \ee^{2 \pi \ii v}}
   \dd u \dd v \\[2mm]
  & = \, \myfrac{1}{2} \bigl( \fm (1 + y^2 - x \ts y)
   + \fm (1 + y^2 - x^{-1} y) \bigr)  \, = \,
   \fm (1 + y^2 - x\ts y) \ts .
\end{split}
\]
Now, one clearly has 
\[
   \fm (1 + y^2 - x\ts y) \, = \, 
   \fm (1 + y^2 + x\ts y) \, = \,
   \fm (1 + x + y) \ts ,
\]
via a change of variables $u = u' + \frac{1}{2}$ for the first
identity and the invariance of the Mahler measure under an invertible
linear map with integer coefficients, as in \cite[Exc.~3.1]{EvWa}, which is
given by the matrix $\left(\begin{smallmatrix} 1 & 1 \\ 0 & 2 
\end{smallmatrix}\right)$ in this case.
\exend
\end{remark}

It is well known that the connection between the Mahler measure and
special values of $L$-series has deep roots; in particular, see
\cite{Deninger,B1998}, and \cite{BL} for a survey with several
examples and references. Also, a connection between Mahler measures
and Lyapunov exponents is known from \cite{Deninger-L}.  On the other
hand, our observation shows that these quantities occur in the
spectral theory of dynamical systems as well, in a rather elementary
way, and it seems an interesting problem to analyse this connection
further.  \bigskip

\noindent
\textbf{Acknowledgements}. 
It is our pleasure to thank Franz G\"{a}hler, Michael J.\ Mossinghoff,
Dan Rust and Tom Ward for discussions, and an anomymous referee for a
number of helpful suggestions.  This work was supported by the German
Research Foundation (DFG) through CRC{\ts\ts}1283. The second author
additionally thanks both the University of Bielefeld and the Alfr\'ed
R\'enyi Institute of the Hungarian Academy of Sciences for hosting him
during this research.  

\end{document}